\documentclass[11pt,twoside]{article}

\usepackage{mathrsfs,amsfonts,amsmath,amssymb}
\usepackage{latexsym}
\newtheorem{satz}{Theorem}
\newtheorem{proposition}[satz]{Proposition}
\newtheorem{theorem}[satz]{Theorem}
\newtheorem{lemma}[satz]{Lemma}

\newtheorem{corollary}[satz]{Corollary}
\newtheorem{remark}[satz]{Remark}

\def\T{\mathsf{T}}

\def\F{\mathbb {F}}
\def\E{\mathsf{E}}

\def\d{\delta}
\def\o{\omega}
\def\({\big (}
\def\){\big )}

\def\G{\Gamma}

\def\le{\leqslant}
\def\ge{\geqslant}
\def\_phi{\varphi}
\def\eps{\varepsilon}

\def\Gr{{\mathbf G}}

\def\f{{\mathbb F}}

\def\R{\mathbb {R}}

\def\Q{\mathsf{Q}}

\author{Shkredov I.D.}
\title{ Any small multiplicative sugroup is not a sumset
\footnote{
This work was supported by grant
Russian Scientific Foundation RSF 14--11--00433.}}
\date{}
\begin{document}
\maketitle

\begin{center}
 Annotation.
\end{center}

{\it \small
    We prove that for an arbitrary $\eps>0$  and any 
	multiplicative subgroup $\G \subseteq \F_p$, $1\ll |\G| \le p^{2/3 -\eps}$ 
	there are no sets $B$, $C \subseteq \F_p$ with $|B|, |C|>1$ such that  $\G=B+C$. 	
	Also, we obtain that
	for $1\ll |\G|  \le p^{6/7-\eps}$ and any $\xi\neq 0$ there is no a set $B$ such that  $\xi \G+1=B/B$. 	
}
\\

\section{Introduction}
\label{sec:introduction}

Let $\Gr= (\Gr,+)$ be an abelian group and $B,C \subseteq \Gr$ be two sets. 
In  Additive Combinatorics sets of the form $B+C := \{ b+c~:~b\in B,\, c\in C\}$ are called the {\it sumsets} and studying properties of such sets is a central
problem  
of this field, see \cite{TV}. 
A set $A$ of integers is {\it additively reducible} if $A$ cannot be written as a set of sums $B+C$ unless one of the sets consists of a single element, see \cite{Ostmann}, \cite{Sarkozy_residues}. 
The question of additive irreducibility of integer sequences was posed by Ostmann \cite{Ostmann} back in '56 (see also a modern overview \cite{Elsholtz}). 
Basically, Ostmann interested in reducibility of  classical sets of Number Theory such as the primes numbers, shifted primes and so on. 
S\'ark\"ozy \cite{Sarkozy_residues} (see also \cite{DS_AD}, \cite{GMS} and references therein) extended such problems to the finite field setting, perhaps for the first time suggesting that 
`multiplicatively structured' sets should be additively irreducible.
As a special case of his program, S\'ark\"ozy conjectured that another classical object of Number Theory, namely, multiplicative subgroups of finite fields of prime order \cite{KS1} are additively irreducible, in particular, the set of quadratic residues modulo a prime. 
Despite some progress (see, e.g., \cite{Sh_Sarkozy}, \cite{Shparlinski_AD} and references therein), the conjecture of S\'ark\"ozy remains open and is considered as an important 
question 
in the field.
The connected problems were 
covered 
in  \cite{EH}, \cite{ES}, \cite{GH}, \cite{GK}, \cite{LS}, \cite{Sarkozy_residues_shift} and in many other articles.

In papers \cite{s_dss}, \cite{s_diff}, \cite{SZ} it was realized that this type of questions is connected with so--called  the {\it sum--product phenomenon}, see \cite{ESz}, \cite{TV}, \cite{Bourgain_more}, 
and the corresponding
problem 
in the real setting was completely solved in \cite{SZ}, where it was proved, in particular, 

\begin{theorem}
	There is $\eps>0$ such that for all sufficiently large finite $A\subset \R$ with $|AA| \le |A|^{1+\eps}$ there is no decomposition $A=B+C$ with $|B|, |C| >1$. 
\label{t:SZ}
\end{theorem}

Here, of course, $AA=\{ a_1 a_2 ~:~ a_1,\, a_2\in A\}$ is the {\it product} set.

Let us formulate the 
main result of our paper.

\begin{theorem}
	Let $p$ be a prime number.
	Let $\eps>0$ be an arbitrary number, and $\G$ be a sufficiently large multiplicative subgroup,  $|\G| \ge C(\eps)$, $|\G| \le p^{2/3-\eps}$.
	Then for any $A,B\subseteq \F_p$, $|A|, |B| >1$ one has $\G \neq A+B$. 
\label{t:G_intr}
\end{theorem}

Thus, we completely solved the problem on additive reducibility of multiplicative subgroups of size 
less than $p^{2/3}$.

\bigskip 

In paper \cite{Sarkozy_residues_shift} S\'ark\"ozy studied {\it multiplicative reducibility} of nonzero shifts of multiplicative subgroups $\G$ and proved that it is not possible to multiply arbitrary three sets to obtain such a shift
(actually, S\'ark\"ozy 
had deal with 
a particular case of quadratic residues).
In our article we consider a symmetric situation and prove that $\xi\G+1\neq B/B$ for any $\xi \neq 0$ and an arbitrary set $B$.  




\begin{theorem}
	Let $p$ be a prime number.
	Let $\eps>0$ be an arbitrary number, and $\G$ be a sufficiently large multiplicative subgroup,  $|\G| \ge C(\eps)$, $|\G| \le p^{6/7-\eps}$.
	Then for any $B\subseteq \F_p$, and an arbitrary $\xi \neq 0$ one has $\xi \G+1 \neq B/B$. 
\label{t:G+x_intr}
\end{theorem}

The method of the proof develops the approach from 
\cite{s_dss}---\cite{SZ}
combining with  \cite{Petridis_quadruples}, \cite{MPR-NRS}.
As it was realized in \cite{s_dss} that to prove $A\neq B+C$, $|B|, |C|>1$ one must firstly separate our set $A$  from the "random sumset case".
Roughly speaking, it means that the most difficult case in the proof of $A \neq B+C$ 
is when 
$|B| \sim |C| \sim |A|^{1/2}$ and $B,C$ look  like random sets with the corresponding  density.
Further, one can note that if $B$, $C$ are random sets, then $B+C$ is also 
behaves randomly  
more or less, so we must exploit some {\it nonrandom} properties of the sumset $A$.
In \cite{Sh_Sarkozy} the author used the fact that if $A$ is the set of quadratic residues, then $A$ is so--called the {\it perfect difference set}, 
and the approach from \cite{s_diff}, \cite{SZ} is based on 
the {\it small doubling}  
property of $A$, namely, that $|AA| \ll |A|$. 
Of course a random set has such properties with probability zero. 
In this paper we continue to use latter non--random feature of $A$.
More concretely, if $|AA| \ll |A|$ and, simultaneously, $A=B+C$, then   $A$ 
is both additively and multiplicatively  rich and this contradicts with the sum--product phenomenon, see \cite{ESz}, \cite{TV}, \cite{Bourgain_more}.
Unfortunately, at the moment the sum--product phenomenon 
over  the prime finite fields
is weaker  then in the real setting, so we could not prove an analog of Theorem \ref{t:SZ} in $\F_p$ just copying the arguments of the proof of  Theorem 
\ref{t:SZ} (see the discussion section from article \cite{SZ}).
In our current approach we use some results from \cite{Petridis_quadruples}, \cite{MPR-NRS} which are comparable with the correspondent theorems  for the reals but on the other hand it requires to change the scheme of the prove from \cite{SZ} somehow. 
So, we apply weaker incidence bounds than were used in \cite{SZ} and, in particular, we reprove the results of this paper. 
Interestingly, that both methods of paper \cite{SZ} and the current one use a nontrivial upper bound for the additive energy of multiplicative subgroups and multiplicatively rich sets, see Theorem \ref{t:32/13} below.

The author is grateful to D. Zhelezov for useful discussions.

\section{Definitions and preliminaries}
\label{sec:preliminaries}

	The following notation is used throughout the paper. 
	Let $p$ be a prime number. 
	By $\F_p$ we denote the prime field. 
	The expressions $X \gg Y$, $Y \ll X$, $Y = O(X)$, $X = \Omega(Y)$ all have the same meaning that there is an absolute constant $c > 0$ such that $|Y| \leq c|X|$. 
	All logarithms are to base $2$.

	For sets  $A$ and $B$ from $\F_p$ the {\it sumset} $A + B$ is the set of all pairwise sums  
	$$
		A+B=\{ a + b ~:~ a \in A,\, b\in B \} \,,  
      $$
	and similarly $AB$, $A-B$ denotes the set of {\it products} and {\it differences}, respectively.
	We denote by $|A|$ the cardinality of a set $A$. 
	The {\it additive energy} $\E^+ (A)$, see \cite{TV},  denotes the number of additive  quadruples $(a_1, a_2, a_3, a_4)$ such that $a_1 + a_2 = a_3 + a_4$. 	
	We use representation function notations like $r_{AB} (x)$ or $r_{A+B} (x)$, which counts the number of ways $x \in \F_p$ can be expressed as a product $ab$ or a sum $a+b$ with $a\in A$, $b\in B$, correspondingly. 
	For example, $|A| = r_{A-A}(0)$ and  $\E^{+} (A) = r_{A+A-A-A}(0)=\sum_x r^2_{A+A} (x) = \sum_x r^2_{A-A} (x)$.

	Now we are ready to introduce the main object of our paper. 
	Let $A,B,C,D \subseteq \F_p$ be four sets.
	By $\Q(A,B,C,D)$ we denote the number of {\it collinear quadruples} in $A\times A$, $B\times B$, $C\times C$, $D\times D$.
	If $A=B=C=D$, then we write $\Q(A)$ for $\Q(A,A,A,A)$.
	Recent results on the quantity $\Q(A)$ can be found in \cite{Petridis_quadruples} and \cite{MPR-NRS}.
	It is easy to see (or consult \cite{MPR-NRS}) that 
\begin{equation}\label{def:Q}
	\Q(A,B,C,D)  = \left| \left\{ \frac{b'-a'}{b-a} =  \frac{c'-a'}{c-a} = \frac{d'-a'}{d-a} ~:~ a,a'\in A,\, b,b'\in B,\, c,c'\in C,\, d,d'\in D \right\} \right| 
\end{equation}
\begin{equation}\label{f:Q_E_3}
		=
			\sum_{a,a'\in A} \sum_x r_{(B-a)/(B-a')} (x)  r_{(C-a)/(C-a')} (x)   r_{(D-a)/(D-a')} (x)  \,.
\end{equation}
	Notice that in (\ref{def:Q}), we mean that the condition, say, $b=a$ implies $c=d=b=a$ or, in other words, that all four points $(a,a'), (b,b'), (c,c'), (d,d')$ have the same abscissa.
	More rigorously, the summation in (\ref{f:Q_E_3})  should be taken over 
	$\F_p \cup \{+\infty\}$,
	where $x=+\infty$ means that the denominator in any fraction $x=\frac{b'-a'}{b-a}$ from, say, $r_{(B-a)/(B-a')} (x)$  equals zero. 
	Anyway, it is easy to see that the contribution of the point $+\infty$ is at most $O(M^5)$, where $M=\max\{ |A|, |B|, |C|, |D| \}$, and hence it is negligible (see, say, Theorem \ref{t:Q} below). 
	Further defining a function $q_{A,B,C,D} (x,y)$ (see \cite{MPR-NRS}) as 
\begin{equation}\label{f:def_t,q}
	q_{A,B,C,D} (x,y) :=  \left| \left\{ \frac{b-a}{c-a} = x,\, \frac{d-a}{c-a} = y ~:~  a\in A,\, b\in B,\, c\in C,\, d\in D \right\} \right| \,,
\end{equation}
	we obtain another formula for the quantity $\Q(A,B,C,D)$, namely, 
$$
	\Q(A,B,C,D) 	= \sum_{x,y} q^2_{A,B,C,D} (x,y) \,.
$$

\bigskip

An optimal  (up to logarithms factors) upper bound for $\Q(A)$ was obtained in \cite{MPR-NRS}, \cite{Petridis_quadruples}.

\begin{theorem}
	Let $A\subseteq \F_p$ be a set.
	Then
$$
	\Q(A) = \frac{|A|^8}{p^2} + O(|A|^5 \log |A|) \,.
$$
	In particular, if $|A| \le p^{2/3}$,  then $\Q(A) \ll |A|^5 \log |A|$. 
\label{t:Q}
\end{theorem}

We need in a simple lemma about a generalization of the quantity $\Q(A)$.
The proof is analogous of  the proof \cite[Lemma 6]{SZ}.

\begin{lemma}
	Let $A,B \subseteq \F_p$ be two sets, 
	$|B| \le |A| \le \sqrt{p}$.  
	Then 
	\begin{equation}\label{f:Q(A,B,C,D)}
		\Q(A,B,A,B) \ll |A|^{5/2}  |B|^{5/2} \log^2 |A| + |A|^3 |B|^2 \,. 
	\end{equation}
\label{l:Q(A,B,C,D)}
\end{lemma}
\begin{proof}
	Write  $L_{i, j}$, $i,j\ge 0$ for the set of lines $\ell$ such that
    $2^{i} \leq |\ell \cap (A \times A)| < 2^{i+1}$ and $2^{j} \leq |\ell \cap (B \times B)| < 2^{j+1}$. 
	Then
    \begin{equation}\label{tmp:16.01.2017_1}
		\Q(A,B,A,B)  \ll \sum^{\log |A|}_{i = 0}\, \sum^{\log |B|}_{j = 0} |L_{i, j}| 2^{2i} 2^{2j} \,.
    \end{equation}
	First of all, consider the lines $\ell$ with $j=0$.
	In other words, each of such a line intersects $B\times B$ exactly at one point. 
	Then we choose a point from $A$, forming a line intersecting $A\times A$, and obtain at most $|A|$ points on each of these lines.
	It gives us at most $|A|^3$ incidences. 
	Another proof  of this fact is the following. 
	Fixing $b,b'$ in (\ref{def:Q}), we have at most $O(|A|^3)$ possibilities for $a,a',c,c' \in A$ such that $\frac{b'-a'}{b-a} =  \frac{c'-a'}{c-a}$.
	Totally, it gives at most $O(|B|^2 |A|^3)$ collinear triples.
	Clearly, the same aruments take place  
	for $i=0$, so we suppose below that $i,j\ge 1$.

Since the number of summands in (\ref{tmp:16.01.2017_1}) is at most $\log^2 |A|$ it is enough to bound each term by $|A|^{5/2}  |B|^{5/2}$. For the sake of notation, denote $k = 2^i$ and
$l = 2^j$, $L = L_{i,j}$,  so that our task is to estimate $|L|k^2 l^2$ where $L$ is the set of lines intersecting $A \times A$ in $k$ (up to a factor of two) points and $B \times B$ in $l$ points (again, up to a factor of two).
Here  $k \geq 2, l \geq 2$.

By the assumption $|B| \le |A| \le \sqrt{p}$. It implies that $2|A|^2/p \le 2 \le k,l$.
The arguments of the proof of Theorem \ref{t:Q} gives us (see \cite[Lemma 14]{MPR-NRS}) that for $k,l\ge 2$ the following incidence estimate holds 
$$
	|L| \ll \min \left( \frac{|A|^5}{k^4} + \frac{|A|^2}{k}, \frac{|B|^5}{l^4} + \frac{|B|^2}{l} \right) \,,
$$
so
$$
	T := k^2 l^2 |L| \ll  \min \left( \frac{ l^2 |A|^5}{k^2} + kl^2 |A|^2,  \frac{k^2 |B|^5}{l^2} + k^2 l |B|^2 \right) \,.
$$
	Since $k\le |A|$, $l\le |B|$, we see that the first terms dominate in the last formula. 
 	Hence multiplying, we get
$$
	T^2 \ll \frac{ l^2 |A|^5}{k^2} \cdot \frac{k^2 |B|^5}{l^2} = |A|^5 |B|^5 
$$
	as required. 
$\hfill\Box$
\end{proof}

\begin{remark}
	It is easy to see that we need in the term $|A|^3 |B|^2$ in (\ref{f:Q(A,B,C,D)}).
	Indeed, take $B=\{0\}$ and $A$ be a multiplicative subgroup in $\F_p$.
	Then one can check that $\Q(A,B,A,B)=|A|^3$ but not $O(|A|^{5/2} \log^2 |A|)$. 
\end{remark}

\bigskip

By $\T(A,B,C)$ denote the number of {\it collinear triples} in $A\times A$, $B\times B$, $C\times C$, where $A,B,C\subseteq \F_p$ are three sets.
It is easy to check that $$\T(A,B,C) = \left| \left\{ \frac{b-a}{c-a} = \frac{b'-a'}{c'-a'}  ~:~ a,a'\in A,\, b,b'\in B,\, c,c'\in C \right\} \right| = \sum_x t^2_{A,B,C} (x) \,,$$
where $t_{A,B,C} (x) = \sum_{a\in A} r_{(B-a)/(C-a)} (x)$.
The support of the function $t_{A,B,C} (x)$ is denoted by $\T[A,B,C]$. 
One has (see \cite{s_diff})
\begin{equation}\label{f:basic_identity}
	\T [A,B,A] = 1-\T[A,B,A] \,. 
\end{equation}
If $A=B=C$, then we write $\T[A]$ for $\T[A,B,C]$ and $\T(A)$ for $\T(A,B,C)$. 
Let us recall a result about $\T[A]$ from \cite{MPR-NRS}.

\begin{theorem}\label{R[A]}
Let  $A \subseteq \F_p$. Then\\ 
$\bullet~$ $|\T[A]| \gg \min\left\{ p, \frac{|A|^{5/2}}{p^{1/2}} \right\}.$ \\
$\bullet~$  $|\T[A]| \gg \min\{p, |A|^{\tfrac{3}{2} + \tfrac{1}{22} - o(1)} \}$, with the $o(1)$ term tending to 0 as $p \to \infty$.\\
$\bullet~$  
	$|\T[A]| \gg \min \left\{ p^{2/3}, \frac{|A|^{8/5}}{\log^{28/15} |A|} \right\}$. 
\end{theorem}


Now let us 
remind 
some results about multiplicative subgroups. 
We need a simplified version of 
Theorem 8 
from \cite{Sh_ineq}.

\begin{theorem}
    Let $p$ be a prime number and $\G \subseteq \F^*_p$ be a multiplicative subgroup,
    $|\G| \le p^{2/3}$.
    Then
    \begin{equation}\label{f:subgroup_energy-}
        \E^{+} (\G)
            \ll
                        |\G|^3 p^{-\frac{1}{3}} \log^{} |\G| + p^{\frac{1}{26}} |\G|^{\frac{31}{13}} \log^{\frac{8}{13}} |\G| \,,
     \end{equation}
	and 
     \begin{equation}\label{f:subgroup_energy}
        \E^{+} (\G)
            \ll
      			|\G|^{\frac{32}{13}} \log^{\frac{41}{65}} |\G| \,,
    \end{equation}
	provided $|\G| < p^{1/2} \log^{-1/5} p$.
\label{t:32/13}
\end{theorem}

\begin{corollary}
	Let $\eps>0$ be a positive real and $\G\subseteq \F_p$ be a multiplicative subgroup, $|\G| \le p^{2/3-\eps}$.
	Then for some $\d(\eps)>0$ one has $\E^{+} (\G) \ll |\G|^{5/2-\d(\eps)}$. 
\label{c:32/13}
\end{corollary}

Because always $\E^{+} (\G) \ge |\G|^4/p$, it follows that $\E^{+} (\G) \gg |\G|^{5/2}$ for $|\G| \sim p^{2/3}$. 
Thus, the constant $2/3$ in Corollary \ref{c:32/13} is optimal.

\bigskip

The last needed result is the main theorem from \cite{V-S}.

\begin{theorem}
    Let $\G\subseteq \f_p$ be a multiplicative subgroup,
    $k\ge 1$ be a positive integer, and $x_1,\dots,x_k$ be different nonzero elements.
    Also, let
     \begin{equation}\label{f:main_many_shifts_cond}
        32 k 2^{20k \log (k+1)} \le |\G|\,, \quad  p \ge 4k |\G|  ( |\G|^{\frac{1}{2k+1}} + 1 ) \,.
    \end{equation}
    Then
    \begin{equation}\label{f:main_many_shifts}
        |\G\bigcap (\G+x_1) \bigcap \dots  \bigcap (\G+x_k)| \le 4 (k+1) (|\G|^{\frac{1}{2k+1}} + 1)^{k+1} \,.
    \end{equation}
	    Further 
    \begin{equation}\label{f:C_for_subgroups}
    	|\Gamma \bigcap (\Gamma+x_1) \bigcap \dots \bigcap (\Gamma+x_k)| 
    		= \frac{|\Gamma|^{k+1}}{(p-1)^k} + \theta k 2^{k+3} \sqrt{p} \,,
	\end{equation}
    where $|\theta| \le 1$.
    The same holds if one replaces $\G$ in (\ref{f:main_many_shifts}) by any cosets of $\G$.
\label{t:main_many_shifts}
\end{theorem}

Thus, 
the theorem above
asserts that
$|\G\bigcap (\G+x_1) \bigcap \dots (\G+x_k)| \ll_k |\G|^{\frac{1}{2}+\alpha_k}$,
provided
$1 \ll_k |\G| \ll_k p^{1-\beta_k}$,
where $\alpha_k, \beta_k$ are some sequences of positive numbers, and $\alpha_k, \beta_k \to 0$, $k\to \infty$.
A little bit better
bounds then in (\ref{f:main_many_shifts_cond}),  (\ref{f:main_many_shifts})
can be found 
in \cite{SSV}.

\section{The proof of Theorem \ref{t:G_intr}}
\label{sec:proof1}

Let us formulate the main technical result of this section.

\begin{proposition}
	Let $\G$ be a multiplicative subgroup, and $A,B\subseteq \F_p$ be arbitrary sets such that   
	$|B| \le |A| \le \sqrt{p}$ 
	and for some nonzero
	$\eta_1,\eta_2$ and two sets $\Omega_1, \Omega_2 \subseteq \F_p$, $|\Omega_1|, |\Omega_2|  \le |\G|$ the following holds 
	\begin{equation}\label{f:p_main}
		\T
		[B,A,A] \subseteq \eta_1 \G \cup \Omega_1 \quad   \mbox{ and } \quad \T [A,B,B] \subseteq \eta_2 \G \cup \Omega_2 \,.
	\end{equation}
	Then
\begin{equation}\label{f:main}
	|A|^4 |B|^{4} |\G| \ll \left( \E^{+} (\G)  + \o |\G|^2 + |\G|^2  \right) \cdot \left(  (|A||B|)^{5/2} \log ^2 |A| + |A|^3   |B|^2 \right)   \,, 
\end{equation}
	where $\omega = \max\{ |\Omega_1|, |\Omega_2| \}$. 
\label{p:main}
\end{proposition}
\begin{proof}
	Let $q(x,y) :=q_{A,B,A,B} (x,y)$. 
	Thus,  any pair $(x,y)$ from the support of the function $q(x,y)$ can be represented as $x = \frac{b-a}{a'- a}$, $y = \frac{b'-a}{a'- a}$, where $a,a'\in A$, $b,b'\in B$. 
	It is easy to see that $\frac{1}{x} = \frac{a'-a}{b-a}$ and because (\ref{f:p_main})
	$$
		1-\frac{1}{x} = 
		1- \frac{a'-  a}{b-a} = \frac{b-a'}{b-a} \in 
		\T [B,A,A] \subseteq \eta_1 \G \cup \Omega_1 \,.
	$$
	Hence 
	$x,y \in (1-(\eta_1 \G \cup \Omega_1))^{-1}  \cup \{0\} := S_1$.
	Also, notice that in view of the second condition from (\ref{f:p_main}), we have for $a\neq a'$ that 
	$$
		\frac{x}{y} = \frac{b-a}{b'-a} \in \T[A,B,B] \subseteq 
		\eta_2 \G \cup \Omega_2 := S_2 \,.
	$$
	These inclusions allow us to obtain a good upper bound for size of the support of the function $q (x,y)$. Let  $\sigma:= |\mathrm{supp}\, q|$.
	By Lemma \ref{l:Q(A,B,C,D)}, we get 
\begin{equation}\label{tmp:14.01.2017_1}
	(|B|^2 |A|^2 )^2 \ll \left( \sum_{x,y} q(x,y) \right)^2 \le \sum_{x,y} S_1 (x) S_1 (y) S_2(x/y) \cdot \sum_{x,y} q^2 (x,y) = \sigma \cdot \Q(A,B,A,B)
\end{equation}
$$
	\ll 
		\sigma \left(  (|A||B|)^{5/2} \log ^2 |A| + |A|^3  |B|^2 \right) \,.
$$
	In the first inequality of (\ref{tmp:14.01.2017_1}) we have used the condition that $|A|, |B| >1$. 
	It remains to estimate the sum $\sigma$. 
	Let $\sigma'$ be the subsum of $\sigma$, where all variables $x,y,x/y$ do not belong to $(1-\Omega_1)^{-1} \cup \Omega_2 \cup \{0\}$ and let $\sigma''$ be the rest. 
	We have
\begin{equation}\label{f:en_tmp1}
	\sigma' \le \left| \left\{ \frac{1-\eta_1 \gamma_1}{1-\eta_1 \gamma_2} = \eta_2 \gamma ~:~ \gamma_1,\gamma_2,\gamma \in \G \right\} \right| 
	=
	 \left| \left\{ 1-\eta_1 \gamma_1 = \eta_2 \gamma (1-\eta_1 \gamma_2) ~:~ \gamma_1,\gamma_2,\gamma \in \G \right\} \right| 
\end{equation}
$$
	=
	 \left| \left\{ 1-\eta_1 \gamma_1 = \eta_2 \gamma- \eta_1 \eta_2 \gamma'_2 ~:~ \gamma_1,\gamma'_2,\gamma \in \G \right\} \right| 
	=
$$
\begin{equation}\label{f:en_tmp2}
	=
	 |\G|^{-1} \E^{+} (\G, -\eta_1 \G, \eta_2 \G, -\eta_1 \eta_2 \G) \le |\G|^{-1} \E^{+} (\G) \,.
\end{equation}
	Here $\E^{+} (A,B,C,D) = |\{ (a,b,c,d) \in A\times B\times C \times D ~:~ a+ b=c+d\}|$ and the last bound in (\ref{f:en_tmp2})  is a consequence of the H\"older inequality.  
	Finally, using a crude bound for the sum $\sigma''$, we obtain 
$$
	\sigma'' \le 3\o (|\G| + \o) + 3\o  (|\G|+ \o) + 3 (|\G|+\o) \le 12 \o |\G| + 6 |\G| \,.
$$	
This completes the proof.
$\hfill\Box$
\end{proof}

\bigskip 

\begin{remark}
	Using Theorem \ref{t:main_many_shifts} with $k=1$, one can improve the error term  $\o |\G|^2 + |\G|^2$ in  (\ref{f:main}).
	Indeed, for $|\G| < p^{3/4}$, say, one can replace it by  $\o |\G|^{5/3} + |\G|^2$. 
	It gives more room in  inclusions (\ref{f:p_main}) allowing $\o$ be 
	 $|\G|^{5/6-\eps_0}$ for some $\eps_0>0$, see Corollary \ref{cor:main} below. 
\end{remark}

\bigskip

Now we are ready to prove Theorem \ref{t:G_intr} from the introduction.

\begin{corollary}
	Let $\eps>0$ be an arbitrary number, and $\G$ be a sufficiently large multiplicative subgroup,  $|\G| \ge C(\eps)$, $|\G| \le p^{2/3-\eps}$.
	Then for any $A,B\subseteq \F_p$, $|A|, |B| >1$ one has $\G \neq A+B$. 
\label{cor:main}
\end{corollary}
\begin{proof}
	Suppose that for some $A,B\subseteq \F_p$, $|A|, |B| >1$ the following holds $\G = A+B$. 
	Without loss of generality
	 suppose that $|B| \le |A|$.
	Redefining $B=-B$, we see that conditions (\ref{f:p_main}) take place with $\eta_1=\eta_2 =1$ and $\Omega_1 = \Omega_2 = \{ 0 \}$. 
	By \cite{Shparlinski_AD}, \cite{V-S} (or just use Theorem \ref{t:main_many_shifts} above) we know that  $|A| \sim |B| \sim |\G|^{1/2+o(1)}$.
	In particular, it gives us for $\G$ large  enough 
	that 
	$|B| \le |A| \le \sqrt{p}$.
	Thus, applying Proposition \ref{p:main}, we obtain
$$
	|A|^8 |B|^{8} |\G|^2 
			\ll
	\E^{+} (\G)^2 \left( |A|^5 |B|^5  + |A|^6 |B|^4 \right) \log^4 |\G|
			\ll 
		|\G|^{o(1)} \E^{+} (\G)^2 |A|^5 |B|^5 \log^4 |\G| \,.
$$
	By Corollary \ref{c:32/13}, we find some $\d = \d(\eps) >0$ such that $\E^{+} (\G) \ll |\G|^{5/2-\d}$ and hence 
$$
	|A|^3 |B|^3 \ll |\G|^{3-2\d + o(1)} \,. 
$$
	Clearly, if $\G=A+B$, then $|A||B| \ge |\G|$.
	Whence
$$
	|\G|^3 \ll |\G|^{3-2\d + o(1)} \,. 
$$
	and it gives a contradiction for large $\G$. 
	This completes the proof.
$\hfill\Box$
\end{proof}

\begin{remark}
	Actually, our results give an effective bound for sizes of sets $A$, $B$ with $A+B\subseteq \G$. 
	It has  the form 
	$\min\{ |A|, |B| \} \ll |\G|^{1/2-c/2}$, where $c>0$ is an absolute constant  (for similar problems, see \cite{GH}).
\label{r:1_subgr}
\end{remark}

\begin{remark}
	It is easy to see that one can replace the  condition $A-B\subseteq \G$  
	onto $A-B\subseteq \G \cup \{ 0\}$, say, or, generally speaking,  onto $A-B\subseteq \G \bigsqcup \Omega$, 
	where $|\Omega| = O(|\G|^{o(1)})$. 
	Moreover, our arguments take place for sets $A,B$ with $A-B\subseteq \bigsqcup_{j=1}^s \xi_j \G$, where
	$s=O(|\G|^{o(1)})$ and $\xi_j \G$ 
	some cosets of $\G$. 

Notice that for small $\G$ it can be $\G  \bigsqcup \{0\} = A-B$, say. For example (see \cite{LS}), let $p=13$, $\G = \{1,3, 4,9,10,12 \}$, $A=B=\{ 2,5,6 \}$.  
Then one can check that, indeed,   $A-A =  \G   \bigsqcup  \{0\}$.
More generally, for any subgroup $A$, $|A|=2$ or $|A|=3$ one has $\G  \bigsqcup  \{0\} = \xi (A-A)$ for some $\xi$ and some subgroup $\G$ (see \cite{s_dss} or just use a direct calculation).
Thus, the condition $|\G| \gg 1$ is required in the corollary above. 
\label{r:2_subgr}
\end{remark}

The
next 
corollary is connected with the main result from \cite{V-S}, see Theorem \ref{t:main_many_shifts} above. 
Bound (\ref{f:k_int}) below is better than (\ref{f:main_many_shifts}) for very large $k$.

\begin{corollary}
	Let $\G$ be a multiplicative sugroup, $|\G| < p^{1/2} \log^{-1/5} p$. 
	Let also $x_1,\dots,x_k \in \F_p$ be any distinct numbers, $k \le |\G|^{\frac{19}{39}} \log^{-\frac{219}{195}} |\G|$. 
	Then
\begin{equation}\label{f:k_int}
	|(\G + x_1) \cap (\G+x_2) \cap \dots \cap (\G+x_k)| 
		\ll
			k^{-2} |\G|^{\frac{19}{13}} \log^{\frac{41}{65}} |\G| \,. 
\end{equation}
\end{corollary}
\begin{proof}
	Put $A= (\G + x_1) \cap (\G + x_2)  \cap \dots \cap (\G + x_k)$ and $B= \{x_1,\dots,x_k\}$.
	If estimate (\ref{f:k_int}) takes place, then there is nothing to prove.  
	Otherwise, we get $|A| \gg k^{-2} |\G|^{\frac{19}{13}} \log^{\frac{41}{65}} |\G|$.
	Further, we have $A-B \subseteq \G$. 
	Using  Proposition \ref{p:main} and our condition on $k$, namely, $k \le |\G|^{\frac{19}{39}} \log^{-\frac{219}{195}} |\G|$, one has 
$$
	|A|^4 k^4 |\G| \ll \E^{+} (\G) ( |A|^{5/2} k^{5/2} \log^2 |\G| + |A|^3 k^2) \ll \E^{+} (\G) |A|^3 k^2 \,.
$$
	Applying Theorem \ref{t:32/13}, we obtain the required result.
$\hfill\Box$
\end{proof}

\section{The proof of Theorem \ref{t:G+x_intr}}
\label{sec:proof2}


It remains to prove our second result, i.e. Theorem \ref{t:G+x_intr}. 
We begin with a lemma which is parallel to the main results of \cite{Sarkozy_residues_shift}, \cite{Shparlinski_AD} 
and we just repeat the proof from \cite{Shparlinski_AD} (although in Theorem \ref{t:A-A_in_G} below we need in  the case $|\G| \le p^{6/7-\eps}$ only). 

\begin{lemma}
	Let $\G\subset \F_p \setminus \{0\}$ be a multiplicative subgroup, $A,B\subseteq \F_p$ be two sets, $|A\setminus\{0 \}|>1$, $|B\setminus \{0\}|>1$.
	Suppose that for some $x\neq 0$ one has $A/B = \G+x$. 
	Then $|A|, |B| \sim |\G|^{1/2+o(1)}$. 
\label{l:0.5_shift}
\end{lemma}
\begin{proof}
	Dividing by $x$ and redefining $A$, we have $A/B=\xi \G +1$, where $\xi = 1/x$. 
	In other words, for any $a\in A$, $b\in B$, $b\neq 0$ one has 
\begin{equation}\label{tmp:20.01.2017_1}
	\frac{a}{b}-1 = \frac{a-b}{b} \in \xi \G \,.
\end{equation}
	One can assume that $1\in B$ and hence $A\subseteq \xi \G +1$. 
	Put $A' = A\setminus \{0\}$, $B' = B\setminus \{0\}$, 
	and assume, in addition, 	 that $|A|\ge |B|$. 
	In view of (\ref{tmp:20.01.2017_1}) for any
	$b_1, \dots, b_k \in B'$ 
	one has 
	\begin{equation}\label{tmp:17.01.2017_1}
		A \subseteq (b_1 \xi \G + b_1) \cap (b_2 \xi \G + b_2)  \cap \dots \cap (b_k \xi \G + b_k) \,.  
	\end{equation}
	Since $|B'| > 1$, it follows that for some different $b_1,b_2 \in B'$, we get $b^{-1}_1 A - 1, b^{-1}_2 A-1 \subseteq \xi \G$ and hence
	$b_1^{-1} b_2^{-1} (A-b_1) (A-b_2) \subseteq \xi^2 \G$. 
	Fix $\eps>0$ and suppose, firstly, that $|\G| \le p^{1-\eps}$. 
	By a sum--product result from \cite{GS}, \cite{Bourgain_more}, we see $|(A-b_1) (A-b_2)| \gg |A|^{1+\d}$, where $\d = \d (\eps)>0$.
	It immediately implies  
	$|A| \ll |\G|^{1-\d/2}$. 
	Thus,  because, trivially, $|A||B| \ge |\G|$, we obtain $|B| \gg |\G|^{\d/2}$ and  we can use inclusion  (\ref{tmp:17.01.2017_1}) and Theorem \ref{t:main_many_shifts} 
	with $k=|\G|^{\d/2}$ 
	to reach, finally, $|B| \le |A| \ll |\G|^{1/2+o(1)}$. 
	Again, $|A| |B| \ge |\G|$ and hence $|A| \sim |B| \sim |\G|^{1/2+o(1)}$.

	Now let $|\G| > p^{1-\eps}$. 
	Let  $\mathcal{X}$ be the set of all multiplicative characters $\chi$ such that $\chi^d = \chi_0$, where $\chi_0$ is the principal character and $d=(p-1)/|\G| < p^{\eps}$.
	Also, put  $\mathcal{X}^* =  \mathcal{X}\setminus \{0\}$.
	If $|A| \ge |B| > p^{\d}$ (actually, we have $|A| \ge |\G|^{1/2}$ automatically,  since $|A|\ge |B|$ and $|A| |B| \ge |\G|$)
	for some positive $\d$, then taking any $\chi \in \mathcal{X}^*$, 
	we obtain 
	$|\sum_{a\in A'} \sum_{b\in B'} \chi (a-b) \overline{\chi(b)} | = |A'| |B'|$.
	Combining this 
	with  the Karatsuba bound \cite{Karatsuba_chi}
	$$
		\sum_{a\in A} \sum_{b\in B} \alpha(a) \beta(b) \chi(a+b) \ll |A|^{1-1/2\nu} \left( |B|^{1/2} p^{1/2\nu} + |B| p^{1/4\nu} \right)  
	$$
	which takes place for any positive integer $\nu$ and any sequences $\alpha$, $\beta$ such that $\|\alpha \|_\infty \le 1$, $\|\beta \|_\infty \le 1$,  
	we obtain
	$$
		|A| |B| \ll |A|^{1-1/2\nu} \left( |B|^{1/2} p^{1/2\nu} + |B| p^{1/4\nu} \right) \ll |A|^{1-1/2\nu}  |B| p^{1/4\nu} \,.
	$$
	Here we have taken $\nu = \lceil 1/\d \rceil$ and have used $|A| \ge |B| > p^{\d}$.  
	Thus, $|A| \ll \sqrt{p} \ll |\G|^{1/2+o(1)}$ 
	as required. 

	To insure that $|A| \ge |B| > p^{\d}$ we just notice that for any $u\in \xi \G+1$ there is $b\in B'$ such that $ub\in A\subseteq \xi \G+1$.
	In other words, $\xi^{-1} (ub-1) \in \G$.
	We have $\G(x) = \frac{1}{d} \sum_{\chi \in \mathcal{X}} \chi(x)$ and hence by 
	the last inclusion,  
	we get for $\{ b_1,\dots, b_k \} = B' \setminus \{1 \}$ that 
$$
	0 = \sum_{u\in \xi \G +1}\, \prod_{j=1}^k  (1-(\xi \G+1) (u b_j)) =  \sum_{u\in \xi \G +1}\, \prod_{j=1}^k \left( 1- \frac{1}{d} \sum_{\chi \in \mathcal{X}} \chi(ub_j-1) \overline{\chi(\xi)} \right) 
	=
$$
$$
	=
		 \sum_{x\in \F_p \setminus \{0\}}\, \prod_{j=1}^k \left( 1- \frac{1}{d} \sum_{\chi \in \mathcal{X}} \chi((1+\xi x^d)b_j-1) \overline{\chi(\xi)} \right) 
	=
	(p-1) (1-d^{-1})^k  + R \,,
$$
	where the first term in the last formula is just a contribution of the principal character, 
	and the error term $R$ absorbs the rest.
	It is easy to check (or see \cite{Shparlinski_AD})  that  the Weil bound implies $|R| \ll k2^k d  (1-d^{-1})^k \sqrt{p}$. 
	Whence  
	we obtain $|B'| = k \gg \log (\sqrt{p}/d) \gg \log p$.
	In view of \cite{Sarkozy_residues_shift} we can assume that $d\ge 3$.
	For $d\ge 3$ choose $k_* = \lceil \log (\sqrt{p}/d) /\log d \rceil \le k$ elements from $B$. 
	Using 
	inclusion 
	(\ref{tmp:17.01.2017_1}), the choice of the parameter $k_*$ and the previous arguments, we have 
	$$
		|A| \ll p/d^{k_*} + k_* \sqrt{p} \ll dp^{1/2} \log p \cdot (\log d)^{-1} \le p^{1/2+o(1)} \le |\G|^{1/2+o(1)} 
	$$
	 (see \cite{Shparlinski_AD}).
	Whence $|B| \ge |\G|/|A| \ge |\G|^{1/2+o(1)}$ as required.    
$\hfill\Box$
\end{proof}

	\begin{remark}
		Let $A=\{0,1\}$, $B=(\G-1) \setminus \{0\}$ or $A=(\G-1) \setminus \{0\}$, $B=\{0,1\}$. 
		Then it is easy to check that $AB=\G-1$. 
		Thus, we need $|A\setminus \{ 0 \}|, |B\setminus \{ 0 \}| > 1$ in general. 
	\end{remark}


Suppose that $A/B =\G+x$. As in the beginning of the proof of Lemma \ref{l:0.5_shift}
we dividing the last identity by $x$ and redefine $A$  such that 
 $A/B=\xi \G +1$, where $\xi = 1/x$. 
	In other words, for any $a\in A$, $b\in B \setminus \{0\}$ one has $\frac{a}{b}-1 = \frac{a-b}{b} \in \xi \G$.
	Whence for all $b,b'\neq 0$ and  $a\neq b,b'$ the following holds 
	\begin{equation}\label{f:T_shift_1}
		\frac{1/b-1/a}{1/b'-1/a} = \frac{(a-b)b'}{(a-b')b} =   \frac{a-b}{b} :  \frac{a-b'}{b'}  \in \xi \G/\xi \G = \G \,.
	\end{equation}
	Similarly, if  $b\neq 0$  and  $a,a'\neq b$,  then 
	\begin{equation}\label{f:T_shift_2}
		\frac{a-b}{a'-b}  =  \frac{a-b}{b} :  \frac{a'-b}{b} \in \xi \G/\xi \G = \G \,.
	\end{equation}
	In other words,  $\T [A^{-1},B^{-1}, B^{-1}], \T [B,A,A] \subseteq \G \bigsqcup \{ 0 \}$ 
	and using  Proposition \ref{p:main} (with $A=B=A$ or $A=B=A^{-1}$), as well as the proof of Corollary \ref{cor:main} in the symmetric case $A=B$, we obtain that $A/A \neq \xi \G+1$ for $|\G| \le p^{2/3-\eps}$.
	Actually, in the case $A=B$ a stronger result takes place (it is parallel to Theorem 36 from \cite{MPR-NRS}).

\begin{theorem}
    Let $\Gamma \subset \F_p$ be a multiplicative subgroup, and $\xi \neq 0$ be an arbitrary residue.
Suppose that for some $A\subset \F_p$
one has
\begin{equation}\label{cond:A-A_new}
    A/A \subseteq \xi \Gamma +1 \,.
\end{equation}
    If $|\Gamma| < p^{3/4}$, then $|A| \ll |\Gamma|^{5/12} \log^{7/6} |\Gamma|$.
    If $p^{3/4} \le |\Gamma| \le p^{5/6}$, then 
    $
    	|A| \ll p^{-5/8} |\Gamma|^{5/4} \log^{7/6} |\Gamma|
    $.
    If $|\Gamma| \ge p^{5/6}$, then 
	$
    	|A| \ll p^{-1} |\Gamma|^{5/3} \log^{1/3} |\Gamma|
    $.\\
    In particular, for any $\varepsilon>0$ and sufficiently large $\Gamma$, $|\Gamma| \le p^{6/7-\varepsilon}$ the following holds
    $$
        A/A \neq \xi \Gamma + 1 
	\,.
    $$
\label{t:A-A_in_G}
\end{theorem} 
\begin{proof}
    We can assume that $|A \setminus \{ 0 \}|>1$.
Put $\Gamma_* = \Gamma \sqcup \{ 0 \}$ and $R=\T[A]$.
Then in the light of 
	(\ref{f:T_shift_2}), we have
\begin{equation}\label{f:R_inclusion}
    R\subseteq (\xi \Gamma \sqcup \{ 0 \} ) / \xi \Gamma = \Gamma_* \,.
\end{equation}
Suppose that $|\Gamma| <  p^{ 3/4}$. 
Applying (\ref{f:basic_identity}) and formula (\ref{f:main_many_shifts}) with $k=1$, we obtain
$$
    |R| = |R \cap(1-R)| \le |\Gamma_* \cap (1-\Gamma_*)| \le  |\Gamma \cap (1-\Gamma)| + 2 	\ll |\Gamma|^{2/3} \le p^{2/3} \,.
$$
Thus, by the third part of Theorem \ref{R[A]}, we get $|R| \gg |A|^{8/5}/\log^{28/15} |A|$.
Hence
$$
	\frac{|A|^{8/5}}{\log^{28/15} |A|} \ll |\Gamma|^{2/3}
$$
and we obtain the required result.

Now let us suppose that $|\Gamma| \ge p^{3/4}$ but $|\Gamma| \le p^{5/6}$.
In this case, 
by formula (\ref{f:C_for_subgroups}) and the previous calculations, we get
$$
	|R| \ll |\Gamma|^2 / p \le p^{2/3} \,.
$$
Applying the third part of Theorem \ref{R[A]} one more time, we obtain 
$$
	\frac{|A|^{8/5}}{\log^{28/15} |A|} \ll |\Gamma|^2 / p
$$
and it gives the required bound for size of $A$, namely,
\begin{equation}\label{tmp:23.11.2016_1}
	 |A| \ll p^{-5/8} |\Gamma|^{5/4} \log^{7/6} |\Gamma| \,.
\end{equation}

Now suppose that $|\Gamma|\ge p^{5/6}$.
Put $q(x,y)=q_{A,A,A,A} (x,y)$.  
We use the second formula from (\ref{f:def_t,q}) and note that if $q(x,y)>0$, then $x-y \in \Gamma_*$.
By inclusion (\ref{f:R_inclusion}), Theorem \ref{t:Q} and the Cauchy--Schwarz inequality, we obtain
$$
    	|A|^8 
	\ll 
		\left( \sum_{x,y} q(x,y) \right)^2
			\le
			 	 \sum_{x,y} q^2 (x,y) \cdot |\mathrm{supp}\, q| 
		\ll
$$
$$ 
		\ll
        	\left( \frac{|A|^8}{p^2} + |A|^5 \log |A| \right)
            	\cdot
    		\left( \sum_{x,y} R (x) R(y) \Gamma_* (x-y) \right)
            	\ll
$$
$$
	\ll
    	\left( \frac{|A|^8}{p^2} + |A|^5 \log |A| \right)
            	\cdot
	\left( \sum_{x} \Gamma_* (x) \Gamma_* (1-x) \sum_y \Gamma_*(y) \Gamma_* (1-y) \Gamma_* (x-y) \right) \,.
$$
It is easy to see that the summand with $x=1$ is negligible in the last inequality.  
Using formula (\ref{f:C_for_subgroups}) with $k=2$ and the condition $|\Gamma|\ge p^{5/6}$, we get
$$
	|A|^8 
    	\ll
        	\left( \frac{|A|^8}{p^2} + |A|^5 \log |A| \right) \cdot \frac{|\Gamma|^5}{p^3} 
            	\ll
                	|A|^5 \log |A| \cdot \frac{|\Gamma|^5}{p^3} \,.
$$
It follows that 
\begin{equation}\label{tmp:23.11.2016_1'}
	|A| \ll p^{-1} |\Gamma|^{5/3} \log^{1/3} |\Gamma| \,.
\end{equation}

Finally, by Lemma \ref{l:0.5_shift}, we get $|A| \sim |\Gamma|^{1/2+o(1)}$, provided $|A\setminus \{0\}|>1$.
Thus, if $|\Gamma| \le p^{6/7-\varepsilon}$, then in view of (\ref{tmp:23.11.2016_1'}) and two another bounds for size of $A$, we have 
 $A/A \neq	\xi \Gamma + 1$ 
 for sufficiently large $\Gamma$.
If $|A\setminus \{0\}|\le 1$, then, clearly, $|\G| \le 2$ and this is a contradiction with the assumption $|\G| \gg 1$. 
$\hfill\Box$
\end{proof}



	\bigskip 

	Because our 
	approach 
	requires just an incidence bound from \cite{MPR-NRS}, \cite{Petridis_quadruples} and  Theorem \ref{t:32/13} which are both have place in $\R$ (see details in \cite{SZ}), we obtain an analog of Theorem \ref{t:SZ} 
	as  well as  Theorem \ref{t:A-A_in_G} in the real setting.

\begin{theorem}
	There is $\eps>0$ such that for all sufficiently large finite $A\subset \R$ with $|AA| \le |A|^{1+\eps}$ 
	there is no decomposition $A+1=B/B$ with $|B\setminus \{0\}| >1$.\\
	In a similar way, let $B \subset \R$ be a set such that  $|B \setminus \{0\}| >1$.
	Then 
	the following holds  
	$$
		|(B/B - 1)(B/B - 1)| \gg |B/B|^{1+c} \,,
	$$
	where $c>0$ is an absolute constant. 
\label{t:SZ'}
\end{theorem}

\bigskip

\noindent{I.D.~Shkredov\\
Steklov Mathematical Institute,\\
ul. Gubkina, 8, Moscow, Russia, 119991}
\\
and
\\
IITP RAS,  \\
Bolshoy Karetny per. 19, Moscow, Russia, 127994\\
and 
\\
MIPT, \\ 
Institutskii per. 9, Dolgoprudnii, Russia, 141701\\
{\tt ilya.shkredov@gmail.com}

\end{document}